\documentclass[11pt,leqno]{amsart}
\usepackage{amsfonts,amssymb,epsfig,latexsym}
\usepackage{amsmath,comment,bm,mathrsfs,braket}
\usepackage[latin1]{inputenc}
\usepackage{color,layout}
\usepackage{ifthen}
\usepackage{a4wide}
\usepackage{hyperref}
\hypersetup{%
  colorlinks = true,
  linkcolor  = blue,
  citecolor    = red
}

\usepackage{xparse}
\makeatletter
\RenewDocumentCommand{\title}{om}{%
   \IfNoValueTF{#1}
     {\gdef\shorttitle{Resonance expansion}}% 
     {\gdef\shorttitle{#1}}%
   \gdef\@title{#2}%
}
\makeatother
\usepackage{graphicx}[dvipdfmx,hiresbb]
\usepackage{color}  
\usepackage{bmpsize}

\newtheorem{theorem}{Theorem}[section]
\newtheorem{lemma}[theorem]{Lemma}
\newtheorem{proposition}[theorem]{Proposition}
\newtheorem{definition}[theorem]{Definition}
\newtheorem{remark}[theorem]{Remark}
\newtheorem{corollary}[theorem]{Corollary}

\def\square{\hbox{\vrule\vbox{\hrule\phantom{o}\hrule}\vrule}}

\newcommand{\be}{\begin{equation}}
\newcommand{\ee}{\end{equation}}
\newcommand{\ben}{\begin{equation*}}
\newcommand{\een}{\end{equation*}}

\newcommand{\til}[1]{\widetilde{#1}}

\numberwithin{equation}{section}

\newcommand{\N}{\mathbb{N}}
\newcommand{\Z}{\mathbb{Z}}
\newcommand{\R}{\mathbb{R}}
\newcommand{\C}{\mathbb{C}}
\newcommand{\T}{\mathbb{T}}

\newcommand{\W}{{\mathcal W}}

\newcommand{\cA}{{\mathcal A}}

\newcommand{\cK}{{\mathcal K}}

\newcommand{\cH}{{\mathcal H}}
\newcommand{\cU}{{\mathcal U}}

\newcommand{\e}{\varepsilon}
\newcommand{\dl}{\delta}
\newcommand{\pphi}{\varphi}

\newcommand{\re}{{\rm Re}\hskip 1pt }
\newcommand{\im}{{\rm Im}\hskip 1pt }
\newcommand{\ord}{{\mathcal O}}

\newcommand{\ope}[1]{{\operatorname{#1}}}
\newcommand{\mc}[1]{{\mathcal{#1}}}
\newcommand{\bu}[1]{{\textbf{\textup{#1}}}}

\newcommand{\refup}[1]{$\ref{#1}$}

\numberwithin{equation}{section}

\begin{document}

\title{Resonances in finitely perturbed quantum walks, resonance  expansion and generic simplicity}
\author{Kenta Higuchi}
\author{Hisashi Morioka}
\author{Etsuo Segawa}

%%%%%%%%%%%%%%%%%%%%%%%%%%%%%%%%%%%%%%%%%%%%

\maketitle

\begin{abstract}
We define resonances for finitely perturbed quantum walks as poles of the scattering matrix in the lower half plane. 
We show a resonance expansion which describes the time evolution in terms of resonances and corresponding Jordan chains. 
In particular, the decay rate of the survival probability is given by the imaginary part of resonances and the multiplicity. 
We prove generic simplicity of the resonances, although there are quantum walks with multiple resonances.
\end{abstract}

%===============================		Introduction		==============================================
\section{Introduction}
Resonances have been studied in many situations including the wave equations and the Schr\"odinger equations, where they describe long time behavior of solutions (see e.g. \cite{DyZw,Ya}). 
Especially the resonance expansion approximates a solution by a finite sum of ``stationary" states corresponding to the resonances while the spectral decomposition needs to integrate the stationary states on the continuous spectrum.

Recently, the scattering problem is also one of the interests in discrete-time quantum walks on $\Z$. 
Quantum walks are time evolutions of a state, a $\C^2$-valued sequence on $\Z$, given by a unitary operator on $l^2(\Z;\C^2)$. 
Thus, the time evolution operator has its spectrum on the unit circle $e^{-i\xi}$ $(\xi\in\R)$ in the complex plane. 
On the continuous spectrum, there are generalized eigenfunctions (see \cite{KKMS,Mo}) in $l^\infty(\Z;\C^2)$. 
A generalized eigenfunction consists of an incident plane wave, a reflected wave and a transmitted wave. 
The scattering matrix is given by the reflection coefficients and the transmission coefficients. 
For example, the properties of the scattering matrix are studied by \cite{FeHi,KKMS,RST,Su,Ti}. 
Moreover, a correspondence to quantum graphs or Schr\"odinger equations \cite{HKSS,Hi,HiSe,Ta}, the resonant-tunneling effect and the local-energy \cite{MMOS,HKKMS} are also studied. 
These problems are considered on the continuous spectrum (unit circle) or its outside where the $l^2$-resolvent operator and scattering matrix are defined analytically. 
In this article, we continue meromorphically the scattering matrix with respect to the spectral parameter to the \textit{inside} of the unit circle. 
Then resonances can be defined as poles of the scattering matrix in the study on quantum walks. 
In addition, we derive resonance expansions in the view of the time evolution of quantum walks, 
and focus on the rate of decay induced by the degeneration of resonances.

It is natural to start with compactly supported potentials in scattering problems for the wave equations and the Schr\"odinger equations \cite{DyZw}. 
We consider the ``finitely perturbed" discrete-time quantum walk on $\Z$ to observe the resonance properties of quantum walks as a first step. Indeed, the finiteness of the perturbation tells us quite detail information on properties of resonances (Proposition~\ref{prop:BoundAlMul}, Proposition~\ref{prop:SymRes}). 
This type of quantum walks appear in the correspondence given in \cite{FeHi,MMOS,Hi} to the Schr\"odinger equation on the line, and also considered in \cite{KKMS,HKKMS}. %since we can compute many objects including the scattering matrix explicitly in this setting. 
The arguments in this article are expected to be generalized to some larger classes of the perturbations considered in \cite{Su,RST}.

In Section~\refup{Sec:DTQWZ}, we set the quantum walk treated here, and define resonances for that. 
We show the existence of the corresponding resonant state, out-going ``stationary" state (Proposition~\ref{prop:EquivRes}), the fact that the geometric multiplicity of a resonance is always one (Proposition~\ref{prop:Gmul}) but the algebraic multiplicity can be larger (Proposition~\ref{prop:Gmul}), and a symmetry of resonances and a relation between the number of resonances and the support of the perturbation (Proposition~\ref{prop:BoundAlMul}). 

In Section~\ref{Sec:ResExp}, we give the resonance expansion and the ``survival probability" of quantum walker. 
More precisely, we describe the time evolution by the finite superposition of (generalized) resonant states, and show that the decay rate of the survival probability is given by the imaginary part of a resonance. 
Here, the generalized resonant states should be considered only in the case that the algebraic multiplicity is larger than the geometric multiplicity $(=1)$. They are the corresponding object to the generalized eigenvectors for a non-diagonalizable matrix. 
The resonance expansion is obtained by the singular value decomposition of the finite-dimensional matrix which is the restriction of the time evolution to the perturbed part. In fact, we show that the time evolution of the perturbed part is given by this matrix. 
In particular, the singular value decomposition is replaced by the eigenvalue decomposition when every resonance is simple (the algebraic multiplicity is also one).
Then the resonance expansion has a much simpler form, and the decay rate become faster. 

In Section~\ref{Sec:GenMul}, we show the fact that the resonances are ``generically" simple (multiple resonances are special). 
To discuss such a topological property, we introduce a topology to the space of finitely perturbed quantum walks. 
We write explicitly the quantization condition for the resonance, and prove the generic simplicity by applying Rouch\'e's theorem. 
However, there certainly exist multiple resonances. We give a simple example by a triple-barrier setting.

\section{Definition and basic properties of resonances}\label{Sec:DTQWZ}
Our objective of this section is to define resonances for a quantum walk, and to observe their basic properties. 
\subsection{Setting}
We first recall the standard definition of two-state quantum walk model on $\Z$ (see e.g. \cite{MMOS}). 
The total Hilbert space is denoted by $\cH:=l^2(\Z;\C^2)\cong l^2(A;\C)$. Here $A$ is the set of arcs of one-dimensional lattice whose elements are labeled by $\{(n;R),\,(n;L);\,n\in\Z\}$, where $(n;R)$ and $(n;L)$ represents the arcs ``from $n-1$ to $n$ (right)" and ``from $n+1$ to $n$ (left)", respectively. 
The time evolution is given by a unitary operator $\cU$ on $\cH$ which is determined by a sequence $(U_n)_{n\in\Z}$ of $2\times 2$ unitary matrices, so called local quantum coins. It is defined by the equality
$$
(\cU\psi)(n)=P_{n+1}\psi(n+1)+Q_{n-1}\psi(n-1)\quad\psi\in\cH,
$$
where we put the matrix valued weights associated with the motion from $n$ to left and to right by
$$
P_n:=\ket{L}\bra{L}U_n,\quad
Q_n:=\ket{R}\bra{R}U_n,
$$
with $\ket{L}:=(1,0)^T,$ $\ket{R}:=(0,1)^T$, $\bra{L}:=(1,0)$, $\bra{R}:=(0,1)$. 
We denote the entries of these matrices by  
\ben
U_n=\begin{pmatrix}a_n&b_n\\c_n&d_n\end{pmatrix},\quad
P_n=\begin{pmatrix}a_n&b_n\\0&0\end{pmatrix},\quad
Q_n=\begin{pmatrix}0&0\\c_n&d_n\end{pmatrix}.
\een
We identify $\cH=l^2(\Z;\C^2)$ and $l^2(A)$ by the following correspondence: 
$$
\left\{
\begin{aligned}
&(n;R)\mapsto\bra{R}\psi(n),\\
&(n;L)\mapsto\bra{L}\psi(n),
\end{aligned}\right.
\quad \text{for}\quad \psi\in l^2(\Z;\C^2).
$$
The time evolution of the free quantum walk is given by the operator $\cU_0$ defined by
$$
(\cU_0\psi)(n)=\ket{L}\bra{L}\psi(n+1)+\ket{R}\bra{R}\psi(n-1)\quad\psi\in\cH.
$$
In other words, every quantum coin of the free quantum walk is the identity matrix $I_2$. 
In this manuscript, we consider finitely perturbed quantum walks: 

\noindent
\textbf{(A1)}
There exists $n_0\in\N$ such that $U_n=I_2$ for any $n\in\Z\setminus\{0,1,\ldots,n_0\}$. 

We put $[n_0]:=\{0,1,\ldots,n_0\}$ and introduce an operator $\cK$ which is a restriction of $\cU$ to $l^2([n_0];\C^2)=\C^{2(n_0+1)}$ determined by 
\be\label{eq:defK}
\begin{aligned}
&(\cK\psi)(n):=(\cU\psi)(n),\quad n=1,2,\ldots,n_0-1,\\
&(\cK\psi)(0):=P_1\psi(1),\quad
(\cK\psi)(n_0):=Q_{n_0-1}\psi(n_0-1).
\end{aligned}
\ee
Note that for any $\bm{v}=(\bm{v}(n))_{n=0}^{n_0}\in\C^{2(n_0+1)}$, we have
\be\label{eq:NormDecreasing}
\|\cK\bm{v}\|_{\C^{2(n_0+1)}}^2+\|P_0\bm{v}(0)\|_{\C^2}^2+\|Q_{n_0}\bm{v}(n_0)\|_{\C^2}^2=\|\bm{v}\|_{\C^{2(n_0+1)}}^2,
\ee
since $U_n$ are unitary, and consequently the operator norm $\|\cK\|$ is bounded by $1$. 

\begin{remark}\label{lem:EquivCond}
Under \bu{(A1)}, the following are equivalent:
\begin{enumerate}
\item 
$\#\{n\in\Z;\,a_n
=0\}\ge2$.
\item There exists an eigenvalue of the operator $\cU$ as an operator on $\cH$.
\item There exists an eigenvalue of $\cK$ with its absolute value $1$. 
\end{enumerate}
\end{remark}

%\begin{comment}
\begin{proof}
We first suppose (3). Let $\bm{v}\in\C^{2(n_0+1)}$ be an eigenvector of $\cK$ associated to an eigenvalue $\lambda$ with $\left|\lambda\right|=1$. 
By definition \eqref{eq:defK} of $\cK$, we have $\bra{R}(\cK\bm{v})(0)=\bra{L}(\cK\bm{v})(n_0)=0$. Then $\cK\bm{v}=\lambda\bm{v}$ implies that 
\be\label{eq:noincoming}
\bra{R}\bm{v}(0)=\bra{L}\bm{v}(n_0)=0.
\ee 
The equality \eqref{eq:NormDecreasing} implies that 
\be\label{eq:nooutgoing}
P_0\bm{v}(0)=Q_{n_0}\bm{v}(n_0)=0.
\ee 
The equalities \eqref{eq:noincoming} and \eqref{eq:nooutgoing} imply that by putting 
\ben
\psi(n)=\bm{v}(n)\quad\text{for}\quad n\in[n_0],\quad
\psi(n)=0\quad\text{for}\quad n\in\Z\setminus[n_0],
\een
$\psi\in\cH$ is an eigenvector associated with the same eigenvalue as $\bm{v}$, that is, (2) holds.
Moreover, when $a_0\neq0$ (resp. $d_{n_0}\neq0$),
$$
P_0\bm{v}(0)
=a_0\bra{L}\bm{v}(0)\ket{L}=0,\quad
(\text{resp. }Q_{n_0}\bm{v}(n_0)=d_{n_0}\bra{R}\bm{v}(n_0)\ket{R}=0)
$$
shows $\bm{v}(0)=0$ (resp. $\bm{v}(n_0)=0$). 
This shows inductively that $\bm{v}=0$ if (1) does not hold. We conclude here that (3) implies (1) and (2). 

We next suppose (1). There exist $0\le n_1<n_2\le n_0$ such that $a_{n_1}=a_{n_2}=0$ and $a_n\neq0$ for any $n_1<n<n_2$. 
Then the restriction of $\cK$ to $l^2(\cA_{n_1,n_2})\cong\C^{2(n_2-n_1)}$ $(\cA_{n_1,n_2}:=\{(n;R),(k;L);\,n_1<n\le n_2,\,n_1\le k<n_2\})$ 
%$(l^2(\{n_1+1,\ldots,n_2-1\};\C^2)\oplus l^2(\{n_1,n_2\};\C))\cong\C^{2(n_2-n_1)}$ 
is a unitary matrix, hence it has $2(n_2-n_1)$ eigenvalues with its absolute value $1$. The corresponding eigenvector can be extended as an eigenvector of $\cK$ corresponding to the same eigenvalue. This implies (3). 

We finally suppose (2). Then there exists an eigenvalue $\lambda$ with $\left|\lambda\right|=1$ and corresponding eigenvector $\psi\in\cH$ since $\cU$ is unitary on $\cH$.
This implies that $\psi$ is supported only on $[n_0]$. In fact, if there exists $n_1\in\Z\setminus[n_0]$ such that $\psi(n_1)\neq0$, say $\psi(-1)=\ket{R}$, then we have $\psi(-n)=\lambda^{n-1}\ket{R}$ and consequently $\psi\notin\cH$. 
The restriction of $\psi$ to $[n_0]$ is an eigenvector of $\cK$ with the same eigenvalue.
\end{proof}

To introduce the scattering matrix, we also assume the following \textbf{(A2)}:

\noindent
\textbf{(A2)}
$a_n\neq0$ holds for any $n\in\Z$.

The operator $\cU$ can be extended as a linear operator on a weighted space $\cH_w:=l^2(\Z;\C^2;(\left|n\right|!)^{-1})$ associated with the inner product 
$$
(\psi^1,\psi^2)_{\cH_w}:=\sum_{n\in\Z}
\frac{1}{\left|n\right|!}(\psi^1(n),\psi^2(n))_{\C^2},\quad\psi^j\in\cH_w.
$$
Note that the operator $\cU$ is no longer unitary on $\cH_w$. 

We first show the unique continuation principle, the existence and uniqueness of generalized stationary state, which is a generalization of the one in \cite{Mo} for real $\xi$ to complex one.

\begin{lemma}\label{lem:UCP}
For any $\xi\in\Xi:=[-\pi,\pi)+i\R$, $\bm{v}\in\C^2$ and $n_1\in\Z$, there uniquely exists $\psi\in\cH_w$ to the equation 
\be\label{eq:EEq}
\cU\psi=e^{-i\xi}\psi
\ee
satisfying $\psi(n_1)=\bm{v}$. 
In particular, the space of the solutions to the equation \eqref{eq:EEq} is two-dimensional.
\end{lemma}

\begin{figure}
\centering
\includegraphics[bb=0 0 277 160, width=6cm]{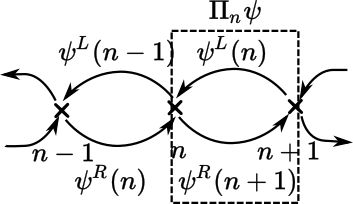}
\caption{Definition of $\Pi_n\psi$}
\label{Fig1}
\end{figure}
\begin{proof}
We introduce $\Pi_{n}:\cH_w\to \C^2$ by
\ben
\Pi_{n}\psi=\bra{L}\psi(n)\ket{L}+\bra{R}\psi(n+1)\ket{R},
\een 
and the local transfer matrix $T_n=T_n(\xi)$ by
\ben
T_n(\xi)=
\begin{pmatrix}
e^{i\xi}/\bar{a}_n&-\bar{c}_n/\bar{a}_n\\
-c_n/d_n&e^{-i\xi}/d_n
\end{pmatrix},\quad
T_n(\xi)^{-1}=
\begin{pmatrix}
e^{-i\xi}/a_n&-b_n/a_n\\
-\bar{b}_n/\bar{d}_n&e^{i\xi}/\bar{d}_n
\end{pmatrix}.
\een
Then the equation \eqref{eq:EEq} is equivalent to the condition
\ben
T_n\Pi_{n}\psi=\Pi_{n-1}\psi \quad\text{for each}\quad n\in\Z.
\een 
We construct a solution $\psi$ satisfying $\psi(n_1)=\bm{v}$ by
\begin{align*}
&\Pi_{n_1}\psi=\bra{L}\bm{v}\ket{L}+e^{i\xi}Q_{n_1}\bm{v}\\
&\Pi_{n}\psi=T_{n+1}T_{n+2}\cdots T_{n_1}\Pi_{n_1}\psi\quad\text{for}\quad n< n_1,\\
&\Pi_{n}\psi=T_n^{-1}T_{n-1}^{-1}\cdots T_{n_1+1}^{-1}\Pi_{n_1}\psi\quad\text{for}\quad n> n_1.
\end{align*}
Conversely, a solution has to satisfy the above equalities, and is unique.
\end{proof}

We give a definition of the scattering matrix by the Jost solutions (see also \cite{Hi}). This coincides with the definition by the spectral decomposition \cite{Mo,Su,RST}. 

\begin{definition}
There exist solutions $\psi_\ope{in}^\pm(\cdot;\xi),$ $\psi_\ope{out}^\pm(\cdot;\xi)$ $\in\cH_w$ to the equation 
satisfying the following \textit{incoming} and \textit{outgoing} property at $\pm\infty$:
\ben
\begin{aligned}
&\psi_\ope{in}^-(n;\xi)=e^{in\xi}\ket{R},\quad
\quad
\psi_\ope{out}^-(n;\xi)=e^{-in\xi}\ket{L}\quad n\le-1,
\\
&\psi_\ope{in}^+(n;\xi)=e^{-in\xi}\ket{L},\quad
\quad
\psi_\ope{out}^+(n;\xi)=e^{in\xi}\ket{R}\quad n\ge n_0+1.
\end{aligned}
\een
We call them the Jost solutions. 
We define the scattering matrix $\Sigma=\Sigma(\xi)$ as a $2\times2$-matrix valued function by
\ben
(\psi_\ope{in}^+(\xi),\psi_\ope{in}^-(\xi))
=(\psi_\ope{out}^-(\xi),\psi_\ope{out}^+(\xi))\Sigma(\xi).
\een
We call the $(1,1)($resp. $(2,2))$-entry left $($resp. right$)$ transmission coefficient, and denote it by $t_-$ $($resp. $t_+)$. 
We call the $(1,2)($resp. $(2,1))$-entry left $($resp. right$)$ reflection coefficient, and denote it by $r_-$ $($resp. $r_+)$. 
\end{definition}

\subsection{Definition of resonances}
We here define resonances by the meromorphic continuation of the scattering matrix.

\begin{theorem}\label{thm:holSmrx}
The scattering matrix $\Sigma(\xi)$ is unitary for real $\xi$,  holomorphic near $\Xi_+=[-\pi,\pi)+i[0,+\infty)$, and meromorphic on $\Xi$. 
The number of the poles of scattering matrix is at most $2n_0$. In particular, $\left|t_\pm\right|^2+\left|r_\pm\right|^2=1$ holds for real $\xi$, where the signs $\pm$ are arbitrarily chosen.
\end{theorem}

To prove this theorem, we show the following lemma.

\begin{lemma}
Let $\psi_1,\psi_2\in\cH_w$ be two solutions to the equation \eqref{eq:EEq}. We denote the determinant by
\ben
\W_n(\psi_1,\psi_2):=\det(\Pi_{n}\psi_1,\Pi_{n}\psi_2)
\een
for each $n\in\Z$. Then the absolute value $\left|\W_n(\psi_1,\psi_2)\right|$ does not depend on $n\in\Z$. Moreover, we have 
\be\label{eq:SMbyWron}
\Sigma(\xi)=\frac{1}{\W_n(\psi_\ope{out}^-,\psi_\ope{out}^+)}
\begin{pmatrix}
\W_n(\psi_\ope{in}^+,\psi_\ope{out}^+)
&\W_n(\psi_\ope{in}^-,\psi_\ope{out}^+)\\
\W_n(\psi_\ope{out}^-,\psi_\ope{in}^+)
&\W_n(\psi_\ope{out}^-,\psi_\ope{in}^-)
\end{pmatrix}
\ee
for any $n\in\Z$ provided $\W_n(\psi_\ope{out}^-,\psi_\ope{out}^+)\neq0$. 
\end{lemma}
\begin{proof}
We immediately see that $\det T_n(\xi)$ is independent of $\xi\in\Xi$. Since $U_n$ is unitary, we have $\left|\det T_n(\xi)\right|=1$ for any $n\in\Z$, and consequently
\ben
\left|\det(\Pi_{n}\psi_1,\Pi_{n}\psi_2)\right|
=\left|\det(\Pi_{n-1}\psi_1,\Pi_{n-1}\psi_2)\right|\quad
\text{for any}\quad n\in\Z.
\een
Since we have
\ben
(\Pi_{n}\psi_\ope{in}^+(\xi),\Pi_{n}\psi_\ope{in}^-(\xi))
=
(\Pi_{n}\psi_\ope{out}^-(\xi),\Pi_{n}\psi_\ope{out}^+(\xi))
\Sigma(\xi)
\een
for any $n\in\Z$, the formula \eqref{eq:SMbyWron} follows.
\end{proof}

\begin{remark}
If we assume $a_n=d_n$ for each $n\in\Z$, it follows that the determinant of $T_n$ is one, and that $\W_n$ does not depend on $n\in\Z$. Then this plays the role of the Wronskian of two solutions to the equation. In fact, by the correspondence given in \cite{Hi} between the solutions to the Schr\"odinger equation and the one to the equation \eqref{eq:EEq}, this coincides with the Wronskian of the corresponding pair of solutions to the Schr\"odinger equation.
\end{remark}

\begin{proof}[proof of Theorem \refup{thm:holSmrx}] 
We see that each element of $T_n(\xi)$ depends analytically on $\xi\in\Xi$, and hence the Jost solutions are $\cH_w$-valued analytic functions on $\xi\in\Xi$. 
This means that if $\W(\psi_\ope{out}^-,\psi_\ope{out}^+)$ does not vanish at a point $\xi_0$, then each element of the scattering matrix is analytic near $\xi_0$. Here $\W_n(\psi_\ope{out}^-,\psi_\ope{out}^+)$ vanishes if and only if $(\psi_\ope{out}^-,\psi_\ope{out}^+)$ is linearly dependent. 
Suppose here that there exists $\xi_0\in\Xi$ such that $(\psi_\ope{out}^-,\psi_\ope{out}^+)$ is linearly dependent when $\xi=\xi_0$. Then the restriction of $\psi_\ope{out}^-$ to $[n_0]$ is an eigenvector of $\cK$ associated to the eigenvalue $e^{-i\xi_0}$. 
According to Lemma \ref{lem:EquivCond} with the assumption \textbf{(A2)}, $\xi_0$ does not belong to $\Xi_+$ since $\left|e^{-i\xi_0}\right|=e^{\im\xi_0}<1$ provided $\xi_0\in\Xi\setminus\Xi_+$. 
The number of the eigenvalues of $\cK$ is at most $2(n_0+1)$. 
In addition, $0$ is an eigenvalue of $\cK$ and its multiplicity is at least 2. 
In fact, $\bm{v}_0, \bm{v}_{n_0}\in\C^{2(n_0+1)}$ given by $\bm{v}_0(n):=\dl_0(n)(d_0\,-c_0)^T$, $\bm{v}_{n_0}(n):=\dl_{n_0}(b_{n_0}\,-a_{n_0})^T$ are two eigenvectors. 
%However, it is not difficult to see that $0$ is an eigenvalue of $\cK$ and its multiplicity is at least 2. 
\end{proof}

\begin{definition}\label{Def:Res}
We define the resonances of the quantum walk produced by $\cU$ satisfying \bu{(A1)} and \bu{(A2)} as the poles $\xi\in\Xi$ of the determinant of the scattering matrix. We denote the set of resonances of the quantum walk generated by the operator $\cU$ by $\ope{Res}(\cU)$.
\end{definition}

\subsection{Some basic properties}
We briefly see some basic properties of resonances. 
We first give some equivalent definitions of resonances.
Put the transfer matrix $\T:=T_0T_1\cdots T_{n_0}$ ($T_n=I_2$ for $n\notin[n_0]$). 
By definition, the Wronskian $\W_{-1}(\psi_\ope{out}^-,\psi_\ope{out}^+)$ at $n=-1$ can be computed by
\ben%\label{eq:WronLD}
\W_{-1}(\psi_\ope{out}^-,\psi_\ope{out}^+)
=\det(\Pi_{-1}\psi_\ope{out}^-,\T
\Pi_{n_0}\psi_\ope{out}^+)
=e^{i(n_0+2)\xi}\T_{22},
\een
where $\T_{22}$ stands for the $(2,2)$-entry of $\T$. 
According to the argument in the proof of Theorem \ref{thm:holSmrx} with the above fact, we obtain the following proposition.
\begin{proposition}\label{prop:EquivRes}
For $\xi\in\Xi$, the follows are equivalent:
\begin{enumerate}
\item $\xi$ is a resonance.
\item There exists an outgoing solution $\psi\in\cH_w$ to the equation \eqref{eq:EEq}.
\item $e^{-i\xi}$ is a $($non-zero$)$ eigenvalue of $\cK$.
\item $\T_{22}$ vanishes at $\xi$. 
\end{enumerate}
We say a state $\psi\in\cH_w$ is \textit{outgoing} if there exists $N_0\ge0$ such that $\bra{R}\psi(-n)=\bra{L}\psi(n_0+n)=0$ holds for any $n\ge N_0$. Here, the outgoing solution satisfies this condition with $N_0=0$. 
\end{proposition}

We will show in Appendix~\refup{A1} that the condition (3) in Proposition~\refup{prop:EquivRes} is equivalent to that $e^{-i\xi}$ is a pole of the \textit{outgoing resolvent}. 
We call an outgoing solution $\pphi\in\cH_w$ to \eqref{eq:EEq} a \textit{resonant state} corresponding to the resonance $\xi$. 
We define the geometric multiplicity of a resonance as the dimension of the space of the associated resonant states which coincides with the dimension of the eigenspace associated with the eigenvalue $e^{-i\xi}$ of $\cK$.

The following is a consequence of Lemma~\refup{lem:UCP}.
\begin{proposition}\label{prop:Gmul}
The geometric multiplicity of each resonance is one.
\end{proposition}

By the above fact, we need to consider only the algebraic multiplicity which is defined as the multiplicity of the zero $\xi$ of $\W(\psi_\ope{out}^-,\psi_\ope{out}^-)$ which coincides with the dimension of the generalized eigenspace associated with the eigenvalue $e^{-i\xi}$ of $\cK$. 
We denote it by $m_R(\xi)$:
$$
m_R(\xi)=\ope{rank}\oint_{\xi} (e^{-i\zeta} I-\cK)^{-1}d\zeta
=\oint_\xi\frac{\T_{22}'(\zeta)}{\T_{22}(\zeta)}d\zeta,
$$ 
where the integrals are over a small circle. 
The algebraic multiplicity is not always one but generically one. 
We will give a simple example of a resonance with $m_R(\xi)>1$,  and show the generic simplicity in Sec.~\ref{Sec:GenMul}. 
We here give the upper bound of the multiplicity.

\begin{proposition}\label{prop:BoundAlMul}
The algebraic multiplicity of each resonance is at most $n_0$.
\end{proposition}
\begin{proof}
By induction, we see that there exists a polynomial $p$ of degree $n_0$ such that $e^{-(n_0+1)\xi}\T_{22}=e^{-2i\xi}p(e^{-2i\xi})$. Since $e^{-(n_0+1)\xi}$ does not vanish, $\xi\in\Xi$ is a resonance if and only if $e^{-2i\xi}$ is a zero of $p$, and the algebraic multiplicity coincides with the multiplicity of $e^{-2i\xi}$ as a zero of $p$.
\end{proof}

The equality $e^{-2i\xi}=e^{-2i(\xi+\pi)}$ for any $\xi\in\Xi$ with the above argument shows the following fact:
\begin{proposition}\label{prop:SymRes}
There exist at most $2n_0$ resonances, where each resonance is counted as many times as its multiplicity. 
$\xi\in[-\pi,0)+i\R$ is a resonance if and only if $\xi+\pi$ is a resonance. 
\end{proposition}

\begin{remark}\label{rem:Redk}
This tells that the resonances appear periodically with its period $\pi$. 
This fact can be generalized to the cases where $U_n=I_2$ except $n\in k\Z\cap[n_0]$ $(k\in\N)$. In that case, the period will be $\pi/k$.
\end{remark}

We have seen that each eigenvector of $\cK$ can be extended to a  resonant state. Similarly, each generalized eigenvector can be extended to a generalized resonant state. The structure of the Jordan chains are also conserved. 
\begin{proposition}\label{prop:genRS}
Let $\xi\in\Xi$ be a resonance and $m:=m_R(\xi)$. Then there uniquely exist outgoing states $\pphi^1,\pphi^2,\ldots,\pphi^m\in\cH_w$ such that 
$$
(\cU-e^{-i\xi})^k\pphi^l=\pphi^{l-k}\quad(1\le k\le l\le m)
$$ 
holds with $\pphi^0=0$ %, $\|\pphi^k\|_{l^2([n_0];\C^2)}=1$ 
and $(\pphi^k|_{[n_0]},\pphi^l|_{[n_0]})_{\C^{2(n_0+1)}}=\dl_{k,l}$ $($Kronecker's delta$)$.
\end{proposition}

\begin{proof}
We denote the normalized resonant state by $\pphi^1$ and its restriction to $[n_0]$ by $\bm{v}^1\in\C^{2(n_0+1)}$. 
Then $\bm{v}^1$ is a normalized eigenvector of $\cK$, and is a basis of the eigenspace. 
By the Jordan normal form of the matrix $\cK$, there exist $\bm{v}^2,\ldots,\bm{v}^m\in\C^{2(n_0+1)}$ such that $(\cK-e^{-i\xi})^k\bm{v}^l=\bm{v}^{l-k}$ with $(\bm{v}^k,\bm{v}^l)_{\C^{2(n_0+1)}}=\dl_{k,l}$. Note that we have $\bra{R}\bm{v}^l(0)=\bra{L}\bm{v}^l(n_0)=0$ for any $l$. Then we extend them inductively, that is, we give $\pphi^k(-n)$ and $\pphi^k(n_0+n)$ for $k>1$, $n\ge1$ by 
\ben
\begin{aligned}
&
\pphi^k(-n)=e^{i\xi}(a_{-n+1}\pphi^k(-n+1)-\pphi^{k-1}(-n)),\\
&
\pphi^k(n_0+n)=e^{i\xi}(d_{n_0+n-1}\pphi^k(n_0+n-1)-\pphi^{k-1}(n_0+n)).
\end{aligned}
\een
In particular, $\pphi^k$ is outgoing. 
\end{proof}

\section{Resonance expansions}\label{Sec:ResExp}
An important application of the resonance is the following resonance expansion which is an analogue of the one for the wave equation and the Schr\"odinger equation.
\subsection{results}
We state the theorems on resonance expansions. We will prove them after a demonstration in the simplest case. 
We introduce vector subspaces $\cH^\ope{out}$, $\cH^\ope{in}$ of $\cH_w$ by
\begin{align*}
&
\cH^\ope{out}:=\bigcup_{N\ge0}\cH^\ope{out}_N,\quad
\cH^\ope{out}_N=\{\psi\in\cH_w;\,\bra{R}\psi(-n)=\bra{L}\psi(n_0+n)=0\text{ for }n\ge N\},\\
&
\cH^\ope{in}:=\bigcup_{N\ge0}\cH^\ope{in}_N,\quad
\cH^\ope{in}_N=\{\psi\in\cH_w;\,\bra{L}\psi(-n)=\bra{R}\psi(n_0+n)=0\text{ for }n\ge N\},
\end{align*}
%$$\cH_c:=c_{00}(\Z;\C^2)=\{\psi\in\cH;\,\psi(n)=0\text{ for }\left|n\right|\ge\exists N_0\},$$
and denote by $\nu(\psi)\in\N=\{0,1,\ldots\}$ for $\psi\in\cH^\ope{out}$ the length of incoming support, that is, the minimum number which satisfies $\psi\in\cH^\ope{out}_\nu$. %$\bra{R}\psi(-n)=\bra{L}\psi(n_0+n)=0$ for any $n\ge\nu$. 
Recall that we always assume \textbf{(A1--2)}. 
The time evolution of an initial state $\psi_0\in\cH^\ope{out}$ can be expanded into a sum of ``generalized resonant states" given in Proposition~\refup{prop:genRS} as follows:
\begin{theorem}\label{thm:REmain}
%Assume \bu{(A1--2)}. 
Let $\{\xi_j;\,0\le j\le 2J\}=\ope{Res}(\cU)$ be the set of resonances, and let $\{\pphi_j^k;\,0\le j\le 2J,\,1\le k\le m_R(\xi_j)\}\subset\cH_0^\ope{out}$ be the corresponding normalized resonant states and the states given in Proposition~\refup{prop:genRS}. There exists $\{\phi_j^k\}\subset\cH^\ope{in}_0$ such that for any initial state $\psi_0\in\cH^\ope{out}$, 
\ben
\begin{aligned}
\psi_t(n)=(\cU^t\psi_0)(n)
&=\sum_{j}^{2J} e^{-it\xi_j}\sum_{k=1}^{m_R(\xi_j)}\braket{\phi_j^k,\psi_0}
\left(\sum_{l=0}^k\begin{pmatrix}t\\t-l\end{pmatrix}e^{il\xi_j}\pphi_j^{k-l}(n)\right)\\
&=\sum_{j}^{2J} e^{-it\xi_j}\sum_{k=1}^{m_R(\xi_j)}\pphi_j^k(n)\sum_{l=k}^{m_R(\xi)}e^{i(l-k)\xi_j}\begin{pmatrix}t\\t-(l-k)\end{pmatrix}\braket{\phi_j^l,\psi_0},
\end{aligned}
\een
holds for any $-(t-\nu(\psi_0))\le n\le t+n_0-\nu(\psi_0)$. 
Here, the coefficients $\braket{\phi_j^k,\psi_0}\in\C$ is given linearly by the finite sum 
\ben%\label{eq:ExpCoeff1}
\braket{\phi_j^k,\psi_0}
=\sum_{n\in\Z}(\phi_j^k(n),\psi_0(n))_{\C^2}
=\sum_{n=-\nu(\psi_0)}^{n_0+\nu(\psi_0)}(\phi_j^k(n),\psi_0(n))_{\C^2},
\een
where $(u,v)_{\C^2}=u^*v=\bar{u}^T v$ is the inner product in $\C^2$, and 
$$
\begin{pmatrix}t\\s\end{pmatrix}=\left\{
\begin{aligned}
&\frac{t!}{s!(t-s)!} \quad&(t\ge s\ge0),\\
&0\quad&(s<0),
\end{aligned}\right.
$$ 
stands for the binomial coefficient.   
\end{theorem}

This gives us the following decay rate of survival probability:
\begin{corollary}\label{cor:DR}
Let $M:=\max_{\xi\in\ope{Res}(\cU)}\left|e^{-i\xi}\right|$ and $m:=\max_{\xi\in\ope{Res}(\cU)}m_R(\xi)$ $(0\le M<1,\, 1\le m\le n_0)$. For any initial state $\psi_0\in\cH^\ope{out}$, there exists $C>0$ such that 
\ben
\|\psi_t\|_{l^2([n_0];\C^2)}\le Ct^{m-1}M^t,\quad \psi_t=\cU^t\psi_0,\quad t\ge1.
\een
Moreover, for any initial state $\psi_0\in\cH^\ope{out}$, there exist $0\le M_0\le M$, $1\le m_0\le m$ and $C_0>0$ such that
\ben
\|\psi_t\|_{l^2([n_0];\C^2)}\le C_0t^{m_0-1}M_0^t,\quad  t\ge1.
\een
\end{corollary}

We will see in Sec.~\refup{Sec:GenMul} that resonances are ``generically" simple. 
Then, the resonance expansion is much simpler:
\begin{corollary}
Assume %\textbf{(A1--2)} and 
that every resonance is simple. 
Let $\{\xi_j\}_{j=1}^{2J}=\ope{Res}(\cU)$ be the set of resonances, and $\{\psi_j\}_{j=1}^{2J}\subset\cH_0^\ope{out}$ be the corresponding normalized resonant states, i.e., $\|\pphi_j\|_{l^2([n_0];\C^2)}=1$. 
Then there exists the set $\{\phi_j\}_{j=1}^{2J}\subset\cH_0^\ope{in}$ of incoming states such that for any initial state $\psi_0\in\cH^\ope{out}$, 
\ben%\label{eq:ResExp1}
\psi_t(n)=(\cU^t\psi_0)(n)=\sum_j e^{-it\xi_j}\braket{\phi_j,\psi_0}\pphi_j(n)
\een
holds for any $-(t-\nu(\psi_0))\le n\le t+n_0-\nu(\psi_0)$. 
In particular, we have the estimates 
\ben%\label{eq:DEsti1}
\|\psi_t\|_{l^2([n_0];\C^2)}
\le \sum_j \left|e^{-i\xi_j}\right|^t\left|\braket{\phi_j,\psi_0}\right|
\quad\text{for }t\ge\nu(\psi_0),
\een 
and
\be\label{eq:DEsti2}
%\begin{aligned}
\|\psi_t\|_{l^2([n_0];\C^2)}
%&
\le \sum_j \left|e^{-i\xi_j}\right|^t
\left(\sum_{n\in\Z}\left|(\phi_j(n),\psi_0(n))_{\C^2}\right|\right)
%\\&
%\le\sum_j \left|e^{-i\xi_j}\right|^{2t}
%\left(\sum_{n\in\Z}\|\phi_j(n)\|^2\|\psi_0(n)\|^2\right)
\quad\text{for }t\ge0.
%\end{aligned}
\ee
\end{corollary}
\begin{remark}
Let $\psi_0\in\cH_w$. If the sequences $(\gamma_{j,l}^k)_{l\in\N}$ $(1\le j\le2J,\,1\le k\le m_R(\xi_j))$ given by
\ben
\gamma_{j,l}^k:=\sum_{n=-l}^{n_0+l}(\phi_j^k(n),\psi_0(n))_{\C^2}
\een
converge to $\gamma_j^k\in\C$ as $l\to+\infty$, we have 
\ben
\Bigl\|\psi_t(n)
-\sum_{j}^{2J} e^{-it\xi_j}\sum_{k=1}^{m_R(\xi_j)}\gamma_j^k
\left(\sum_{l=0}^k\begin{pmatrix}t\\t-l\end{pmatrix}e^{il\xi_j}\pphi_j^{k-l}(n)\right)
\Bigr\|_{\C^2}\to0\quad\text{as }t\to+\infty
\een
for each $n\in\Z$. 
Moreover, Corollary~\refup{cor:DR} is also true, and the estimate \eqref{eq:DEsti2} still holds if the convergence is absolute.
This is satisfied for example when there exist $\dl>0,$ $C>0$ such that
$$
\|\psi_0(n)\|_{\C^2}\le C\left|n\right|^{-m-\dl}\min_{\xi\in\ope{Res}(\cU)}\left|e^{-i\xi}\right|^{|n|}\quad \text{for any}\quad n\in\Z.
$$
\end{remark}

\subsection{Demonstration on double barrier problem}\label{sec:DB}
We first show explicitly the resonance expansion in the simplest case, double barrier problem. We assume \bu{(A3)}:

\noindent
\textbf{(A3)}
\textbf{(A1)} holds with $n_0=1$, and $U_0,U_1$ are not diagonal. 

\begin{remark}
There is no resonance when $U_n$'s are diagonal except at one point. 
There are many other settings which satisfy \bu{(A1--2)} and represent the double barrier. 
However, each of them can be reduced to a quantum walk satisfying \bu{(A3)}.
\end{remark}

\begin{proposition}\label{prop:ResDouble}
Assume \bu{(A2)} and \bu{(A3)}. Then resonances are the two roots $\xi_\pm\in\Xi$ of $e^{2i\xi_\pm}=c_0b_1$ where $\re\xi_+\in[0,\pi),$ $\re\xi_-\in[-\pi,0)$. The $($not normalized$)$ resonant state $\pphi^\pm$ associated with $\xi_\pm$ is given by 
\ben
\psi^\pm(n)=
\left\{
\begin{aligned}
&\pm b_1(\pm\lambda)^n\ket{L}\quad&n\le0,\\
&\pm(\pm\lambda)^{2-n}\ket{R}\quad&n\ge1,
\end{aligned}\right.
\een
where we denote by $e^{i\xi_\pm}=\pm\sqrt{c_0b_1}=\pm\lambda$. 
Here, we chose the branch of the square root of $c_0b_1\in\C\setminus[0,+\infty)$ such that the imaginary part is positive. 
\end{proposition}

The above Proposition~\ref{prop:ResDouble} is obtained by computing the eigenvalues and the eigenvectors of the representation matrix of the restricted operator $\cK$ given by
$$
\cK=
\begin{pmatrix}
0&0&a_1&b_1\\
0&0&0&0\\
0&0&0&0\\
c_0&d_0&0&0
\end{pmatrix}.
$$

\begin{proposition}\label{prop:ResExpDoub}
For any initial state $\psi_0\in\cH^\ope{out}$, we have
\be\label{eq:ResExpDoub}
\psi_t(n)=(\cU^t\psi_0)(n)=\sum_{\pm}(\pm\lambda)^t\braket{\phi_\pm,\psi_0}\psi^\pm(n)
\ee
for any $-t+\nu(\psi_0)\le n\le t+1-\nu(\psi_0)$,  where $\phi_\pm\in\cH_0^\ope{in}$ are given by 
\ben
2(\phi_\pm(n))^*=\left\{
\begin{aligned}
&\pm\lambda^{-2}\bra{R}Q_0\quad&(n=0),\\
&b_1^{-1}\lambda^{-1}\bra{L}P_1\quad&(n=1),\\
&\pm d_0(\pm\lambda)^{-|n|+2}\bra{R}\quad&(n<0),\\
&\pm a_1b_1^{-1}(\pm\lambda)^{-n}\bra{L}\quad&(n>1).
\end{aligned}
\right.
\een
In particular, there exists $C>0$ such that 
\be\label{eq:ResDecayDoub}
\|\psi_t\|_{l^2([1];\C^2)}\le C\left|\lambda\right|^t,
\ee
holds for any $t\in\N$. 
The constant $C>0$ is given by $C^2=|b_1|(|b_1|+|c_0|)\sum_{\pm}\left|\gamma_\pm\right|^2$. 
\end{proposition}

\begin{proof}
If the resonance expansion \eqref{eq:ResExpDoub} is true, then the estimate \eqref{eq:ResDecayDoub} is clear since 
$$
\|\psi^\pm\|_{l^2([1];\C^2)}^2=\left|b_1\right|^2+\left|\lambda\right|^2=\left|b_1\right|(\left|b_1\right|+\left|c_0\right|).
$$
Let us prove the expansion. 
By the singular value decomposition and the diagonalization, we obtain 
\begin{align*}
\cK^t&=
\begin{pmatrix}
b_1&-b_1\\0&0\\0&0\\\lambda&\lambda
\end{pmatrix}
\begin{pmatrix}
\lambda^t&0\\
0&(-\lambda)^t
\end{pmatrix}
\left(
\frac{1}{2\lambda^2}
\begin{pmatrix}
c_0&d_0&a_1b_1^{-1}\lambda&\lambda\\
-c_0&-d_0&a_1b_1^{-1}\lambda&\lambda
\end{pmatrix}
\right)\\
&=\sum_{\pm}(\pm\lambda)^t
\chi_{[1]}(\psi^\pm)
(2\lambda^2)^{-1}
\begin{pmatrix}
\pm c_0&\pm d_0& a_1b_1^{-1}\lambda&\lambda
\end{pmatrix}.
\end{align*}
Here, $\chi_{[n_0]}:\cH_w\to\C^{2(n_0+1)}$ is the restriction operator. 
We decompose the initial state $\psi_0$ into three parts
\ben
\psi_0=\psi_0^\ope{comp}+\psi_0^\ope{in}+\psi_0^\ope{out},
\een
with 
\begin{align*}
\begin{aligned}
&\psi_0^\ope{comp}(0)=\psi(0),\quad\psi_0^\ope{comp}(1)=\psi_0(1),\quad
\psi_0^\ope{comp}(n)=\bm{0}\quad(n\neq0,1),\\ 
&\psi_0^\ope{in}(n)=\bra{R}\psi_0(n)\ket{R}\quad(n\le-1),\quad
\psi_0^\ope{in}(n)=\bra{L}\psi_0(n)\ket{L}\quad(n\ge2),\\
&\psi_0^\ope{out}(n)=\bra{L}\psi_0(n)\ket{L}\quad(n\le-1),\ \ 
\psi_0^\ope{out}(n)=\bra{R}\psi_0(n)\ket{R}\quad(n\ge2).
\end{aligned}
\end{align*}
Then the time evolution of $\psi_0^\ope{out}$ is obvious. For each $t$, we have $(\cU^t\psi_0^\ope{out})(n)=0$ for any $-t\le n\le1+t$. We next treat $\psi_t^\ope{comp}=\cU^t\psi_0^\ope{comp}$. Since the restriction $\chi_{[1]}(\psi_t^\ope{comp})$ is given by $\cK^t\chi_{[1]}(\psi_0^\ope{comp})$, we have 
\ben
\chi_{[1]}(\psi_t^\ope{comp})
=\sum_\pm \gamma_\pm^\ope{comp}(\pm\lambda)^t \chi_{[1]}(\psi^\pm)
\een
with 
\ben
\gamma_\pm^\ope{comp}=\frac{\pm1}{2\lambda^2}\bra{R}Q_0\psi_0^\ope{comp}(0)
+\frac{1}{2b_1\lambda}\bra{L}P_1\psi_0^\ope{comp}(1).
\een
It follows that for each $t\in\N$,
\ben
\psi_t^\ope{comp}(n)=\sum_\pm \gamma_\pm^\ope{comp}(\pm\lambda)^t \psi^\pm\quad -t+1\le n\le t.
\een
To treat $\psi_0^\ope{in}$, we consider the time evolution of 
\ben
\psi^{\alpha,\beta}(n)=\alpha\dl_0(n)\ket{R}+\beta\dl_1(n)\ket{L}\in\cH_c\quad \alpha,\beta\in\C.
\een 
Since this is supported only on $[1]$, the time evolution is given by
\ben
(\cU^t\psi^{\alpha,\beta})(n)=\sum_\pm \gamma_\pm^{\alpha,\beta}(\pm\lambda)^t \psi^\pm\quad -t+1\le n\le t
\een
with 
\ben
\gamma_\pm^{\alpha,\beta}=\frac{\pm1}{2\lambda^2}\bra{R}Q_0\psi^{\alpha,\beta}(0)
+\frac{1}{2b_1\lambda}\bra{L}P_1\psi^{\alpha,\beta}(1)
=\frac{\pm d_0}{2\lambda^2}\alpha+\frac{a_1}{2b_1\lambda}\beta.
\een
Note that by setting $\psi_0^k(n):=\bm{1}_{\{-k,1+k\}}(n)\psi_0^\ope{in}(n)$, we have $\cU^k\psi_0^k=\psi^{\alpha_k,\beta_k}$ for $\alpha_k=\bra{R}\psi_0(-k)$, $\beta_k=\bra{L}\psi_0(1+k)$, and consequently, we have 
\ben
\psi_j^k(n)=\sum_\pm \gamma_\pm^{\alpha_k,\beta_k}(\pm\lambda)^{t-k} \psi^\pm\quad -t+k+1\le n\le t-k
\een 
for $t>k$. We obtain the resonance expansion \eqref{eq:ResExpDoub} by the linearity of the time evolution. 
\end{proof}

\subsection{Proof of Theorem~\refup{thm:REmain}}
We then give a proof for the theorems on resonance expansions. The argument is essentially the same as above.
We decompose the initial state $\psi_0$ into three parts
\ben
\psi_0=\psi_0^\ope{comp}+\psi_0^\ope{in}+\psi_0^\ope{out},
\een
with 
\begin{align*}
\begin{aligned}
&\psi_0^\ope{comp}(n)=\psi(n)\quad(n\in[n_0]),\quad
\psi_0^\ope{comp}(n)=\bm{0}\quad(n\notin[n_0]),\\ 
&\psi_0^\ope{in}(n)=\bra{R}\psi_0(n)\ket{R}\quad(n\le-1),\quad
\psi_0^\ope{in}(n)=\bra{L}\psi_0(n)\ket{L}\quad(n\ge n_0+1),\\
&\psi_0^\ope{out}(n)=\bra{L}\psi_0(n)\ket{L}\quad(n\le-1),\ \ 
\psi_0^\ope{out}(n)=\bra{R}\psi_0(n)\ket{R}\quad(n\ge n_0+1).
\end{aligned}
\end{align*}
Then the time evolution of $\psi_0^\ope{out}$ is obvious. 
The restriction of $\psi_0^\ope{comp}$ to $[n_0]$ can be written  uniquely by the linear combination of the generalized eigenvectors of $\cK$ which correspond to the generalized resonant states $\pphi_j^k$. Then we obtain the expansion of the time evolution of $\psi_0^\ope{comp}$. 
For the incoming part $\psi_0^\ope{in}$, by a similar argument as the case of double barrier, it is reduced to the time evolution of 
$$
\psi^{\alpha,\beta}(n)=\alpha\dl_0(n)\ket{R}+\beta\dl_{n_0}(n)\ket{L}\in\cH_c\quad \alpha,\beta\in\C.
$$
This is done similarly as for $\psi_0^\ope{comp}$ since $\psi^{\alpha,\beta}$ is only supported on $[n_0]$. 

\begin{figure}
\centering
\includegraphics[bb=0 0 779 105, width=13.5cm]{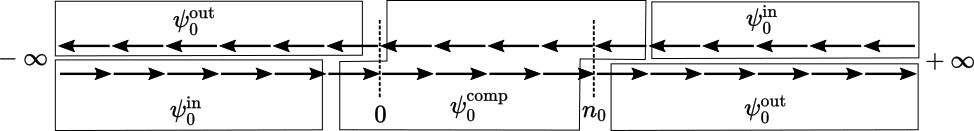}
\caption{Decomposition of $\psi_0$}
\label{Fig2}
\end{figure}

\section{Algebraic multiplicity of resonances}\label{Sec:GenMul}
In the previous section, we have seen that the resonance expansion is much simpler when every resonance is simple, the algebraic multiplicity is one. We have also seen that resonances are simple in the double barrier problem.  
In this section, we first give an example of multiple resonances in the triple barrier problem. 
After that, we show the generic simplicity which tells us that multiple resonances are special.

\subsection{An example for multiple resonances}
Let us consider the triple barrier problem.

\vskip2mm
\noindent
\textbf{(A4)} \textbf{(A1)} holds with $n_0=2$, and $U_0$, $U_1$, $U_2$ are not diagonal. 

\begin{proposition}
Assume \bu{(A2)} and \bu{(A4)}. Then the two roots $\xi_\pm\in\Xi\setminus\Xi_+$ of $e^{-2i\xi}=(c_1b_2+c_0b_1)/2$ are resonances whose algebraic multiplicity is two, if and only if 
\be\label{eq:multicond}
(c_1b_2+c_0b_1)^2+4c_0b_2\det U_1=0.
\ee
\end{proposition}
\begin{proof}
By a straight calculation, we have
\ben
e^{-i\xi}d_0d_1d_2\T_{22}=(e^{-2i\xi})^2-(c_1b_2+c_0b_1)e^{-2i\xi}-c_0b_2\det U_1.
\een
This is a polynomial of degree $2$ of $e^{-2i\xi}$. This has a multiple root if and only if \eqref{eq:multicond} holds.
\end{proof}

The condition \eqref{eq:multicond} is complicated. A necessary condition is given by
\ben
\left|b_1\right|=\frac{2\sqrt{|c_0b_2|}}{|c_0+b_2(c_1/b_1)|}
\ge\frac{2\sqrt{r_0r_2}}{r_0+r_2},\quad
r_n:=\left|b_n\right|=\left|c_n\right|.
\een
The assumption \bu{(A2)} implies $\left|b_1\right|<1$, and  the right hand side is equal to one when $r_0=r_2$. Thus, $r_0\neq r_2$ is required, where 
Let us restrict ourselves to the case that $a_n=d_n=\sqrt{1-r_n^2}$, $b_n=-c_n=r_n$. Then the condition \eqref{eq:multicond} turns into 
\ben
r_1=\frac{2\sqrt{r_0r_2}}{r_0+r_2}.
\een
Then for any $r_0\neq r_2$, there uniquely exists $0<r_1<1$ which satisfies this condtion.  
We give an example:
\ben
r_0=\frac{3}{4},\quad r_1=\frac{12}{13},\quad r_2=\frac{1}{3}.
\een

\subsection{Generic simplicity}
To discuss the ``generic" property, we have to introduce a topology to a class of quantum walks. 
Let $\mc{S}$ and $\mc{T}$ be the sets of $2\times2$ matrices given by
\ben
\begin{aligned}
&\mc{T}=\{e^{i\theta}T;\,T\in\ope{SL}(2),\ T_{11}=\bar{T}_{22},\ T_{21}=\bar{T}_{12},\ \theta\in[0,\pi)\},\\
&\mc{S}=\{S\in\ope{U}(2);\,\ S_{11}\neq0\}.
\end{aligned}
\een
Note that each coin matrix $U_n$ and the scattering matrix for real $\xi$ belong to $\mc{S}$, and that local and global transfer matrices $T_n,\T$ for real $\xi$ belong to $\mc{T}$. For each element $T\in\mc{T}$, $S\in\mc{S}$, there uniquely exists $(p,q,\theta)\in \{(z,w)\in\C^2;|z|^2-|w|^2=1\}\times(\R/2\pi\Z)$ such that $\theta\in[0,\pi)+2\pi\Z$ and
\ben
T=T_{p,q,\theta}:=
e^{i\theta}
\begin{pmatrix}
p&\bar{q}\\q&\bar{p}
\end{pmatrix},
\quad
S=S_{p,q,\theta}:=
\bar{p}^{-1}
\begin{pmatrix}
e^{i\theta}&\bar{q}\\-q&e^{-i\theta}
\end{pmatrix}.
\een
Here, we define an equivalence relation $(p,q,\theta)\sim(-p,-q,\theta-\pi)$ since we have $T_{p,q,\theta}=T_{-p,-q,\theta-\pi}$ and $S_{p,q,\theta}=S_{-p,-q,\theta-\pi}$. 
Then $\mc{T}$ and $\mc{S}$ are topological spaces induced by the topology of $\{(z,w)\in\C^2;|z|^2-|w|^2=1\}\times(\R/2\pi\Z)$ divided by $\sim$, the product of a subset of $\C^2$ and the torus $\R/2\pi\Z$. 
The topological space $\mc{T}$ forms a topological group with respect to the usual product of matrices, and $\mc{S}$ forms a topological group with respect to the product $*$ and the topology induced by the bijection $\mc{M}:\mc{T}\to\mc{S}$ given by $\mc{M}(T_{p,q,\theta}):=S_{p,q,\theta}$. 
We denote the operator $\cU$ determined by the sequence $(U_n)_{n\in\Z}$ by $\cU((U_n)_{n\in\Z})$, and introduce 
$$
\mc{S}([n_0]):=\{\cU((U_n)_{n\in\Z});\,U_n\in\mc{S}\text{ for }n\in[n_0],\  U_n=I_2\text{ for }n\notin[n_0]\}\cong \mc{S}^{n_0+1}.
$$
Remark that we have the inclusion $\mc{S}([n_0])\subset\mc{S}([n_0+1])$ for any $n_0\in\N$, and $\mc{S}([n_0])$ is closed in $\mc{S}([n_0+1])$. Then the topology for the space $\varinjlim_{n_0\in\N}\mc{S}([n_0])$ of finitely perturbed quantum walks satisfying \bu{(A1--A2)} is determined.

\begin{remark}
For each element $S\in\mc{S}$, we have
$$
S=\begin{pmatrix}a&b\\c&d\end{pmatrix}=S_{p,q,\theta},
\quad
(ad\neq0,\ c=-\bar{d}^{-1}a\bar{b})
$$
for the parameters given by
$p=|a|^{-1}e^{i(\arg ad)/2}$, 
$q=|a|^{-1}\bar{b}e^{i(\arg ad)/2}$, 
$\theta=2^{-1}(\arg (a/d))$. 
Consequently, the inverse of $\mc{M}$ satisfies
\ben
\mc{M}^{-1}(S)=\frac{e^{i(\arg (a/d))/2}}{|a|}
\begin{pmatrix}e^{i(\arg ad)/2}&be^{-i(\arg ad)/2}\\
\bar{b}e^{i(\arg ad)/2}&e^{-i(\arg ad)/2}
\end{pmatrix}
=
\begin{pmatrix}
\bar{a}^{-1}&b/d\\
\bar{b}/\bar{a}&d^{-1}
\end{pmatrix}.
\een
\end{remark}
\begin{remark}
Define the subspaces $\mc{T}_0\subset \mc{T}$ and $\mc{S}_0\subset \mc{S}$ by
\ben
\mc{T}_0=\{T_{p,q,\theta}\in\mc{T};\,\theta=0\},\quad
\mc{S}_0=\{S_{p,q,\theta}\in\mc{S};\,\theta=0\}.
\een
They are enough to consider quantum walks corresponding to the Schr\"odinger operators in the meaning of \cite{Hi}. 
\end{remark}

The following theorem asserts that resonances are generically simple.

\begin{theorem}\label{thm:GenSim}
For each $n_0$, there exists a dense subset $\mathcal{V}\subset \mc{S}([n_0])$ such that for any $\cU\in\mc{V}$, every resonance 
is simple.
\end{theorem}

To prove this theorem, we apply the following \cite[Lemma 2.26]{DyZw}.
\begin{lemma}[\cite{DyZw}]\label{lem:Rouche}
Let $\e\mapsto f_\e(z)$ be a family of functions holomorphic in a complex disc $D(0,r_0)=\{z\in\C;\,|z|\le r_0\}$ satisfying 
\ben
f_\e(z)=z^m-\e+\ord(\e^2)+\ord(\e|z|),\quad \left|z\right|\le r_0.
\een
Then for $\e$ sufficiently small, $f_\e(z)$ has exactly $m$ simple zeros in $D(0,r_0)$.
\end{lemma}

\begin{proof}[proof of Theorem~\refup{thm:GenSim}]
We denote by $\lambda(\xi)=e^{-i\xi}$. 
Let $\cU\in\mc{S}([n_0])$ have resonances $\xi_0,\xi_0+\pi\in\Xi$ whose algebraic multiplicity is $m>1$. 
Then there exist a polynomial $\sigma$ of degree $n_0$ and $\til{\sigma},$ $\til{\tau}$ of degree $n_0-1$ such that 
\ben
\lambda^{n_0+1}\T_{22}=\lambda^2 \sigma(\lambda^2)=
\begin{pmatrix}
\lambda\bar{b}_0/\bar{a}_0&\lambda^2/d_0
\end{pmatrix}
\begin{pmatrix}
\lambda\til{\tau}(\lambda^2)\\
\lambda^2\til{\sigma}(\lambda^2)
\end{pmatrix},
\een
and $\lambda_0^2:=\lambda(\xi_0)^2=e^{-2i\xi_0}$ is a zero of $\sigma$ with its multiplicity $m$. 
We fix $\phi\in[-\pi,\pi)$ and consider a family of perturbed operators $(\cU_{\e,\phi}:=\cU((\til{U}_n(\e,\phi))_{n\in\Z}))_{\e>0}\in\mc{S}([n_0])$ given by $\til{U}_n=U_n$ for $n\in\Z\setminus\{0\}$ and 
\begin{align*}
&
\til{U}_0(\e,\phi)
:=S_{p,q,0}
%:=\frac{1}{1+e^{-i\phi}\e}
%\begin{pmatrix}\sqrt{1+2\e\cos\phi}&e^{-i\phi}\e\\ 
%-e^{i\phi}\e&\sqrt{1+2\e\cos\phi}\end{pmatrix}
*U_0\in\mc{S},\quad
p(\e,\phi)=\frac{1+e^{i\phi}\e}{\sqrt{1+2\e\cos\phi}},\quad
q(\e,\phi)=\frac{\e e^{i\phi}(1+e^{i\phi}\e)}{\sqrt{1+2\e\cos\phi}}.
\end{align*}
Then we have $(p(\e,\phi),q(\e,\phi))\to(1,0)$, that is, $S_{p,q,0}\to I_2$ as $\e\to0_+$, and 
\ben
\lambda^{n_0+1}\T^{(\e,\phi)}_{22}
=k_{\e,\phi}\lambda^2
\left(
\sigma(\lambda^2)+\e(\bar{a}_0^{-1}\alpha_\phi\til{\tau}(\lambda^2)+\lambda^2d_0^{-1}\bar{\alpha}_\phi\til{\sigma}(\lambda^2))
\right),
\een
where we set $\alpha_\phi=e^{i\phi}+\bar{b}_0e^{-i\phi}$, 
$k_{\e,\phi}=(1+2\e\cos\phi)^{-1/2}$. 
By the Taylor expansion at $\lambda(\xi_0)^2$, we have
\ben
\frac{\lambda^{n_0-1}}{k_{\e,\phi}}\T^{(\e,\phi)}_{22}
=z^m+\e(\bar{a}_0^{-1}\alpha_\phi\til{\tau}(\lambda_0^2)+\lambda_0^2d_0^{-1}\bar{\alpha}_\phi\til{\sigma}(\lambda_0^2))
+\ord(z^{m+1})
+\ord(\e z)
\een
with $z=\lambda^2-\lambda_0^2$. Since the determinant of each $T_n(\xi)$ does not vanish, $\til{\sigma}$, $\til{\tau}$ does not vanish at the same time, and we fix $\phi$ so that $\bar{a}_0^{-1}\alpha_\phi\til{\tau}(\lambda_0^2)+\lambda_0^2d_0^{-1}\bar{\alpha}_\phi\til{\sigma}(\lambda_0^2)$ does not vanish. 
The implicit function theorem shows that Lemma~\ref{lem:Rouche} is applicable to the family $f_\e(z)=\frac{\lambda^{n_0-1}}{k_{\e,\phi}}\T^{(\e,\phi)}_{22}$, and the simplicity of the resonances of $\cU_{\e,\phi}$ near $\xi_0$ for small $\e>0$ follows. 
Note that the choice of $\phi\in[-\pi,\pi)$ only depends on $\xi_0$. 
Since the number of the resonances is at most $2n_0$, we can take $\phi\in[-\pi,\pi)$ such that every resonance for $\cU_{\e,\phi}$ is simple for small $\e>0$. The fact that $\cU_{\e,\phi}\to\cU$ as $\e\to0_+$ implies the density. 
\end{proof}

\section*{Acknowledgements}
H.M.was supported by the Grant-in-Aid for Young Scientists Japan Society for the Promotion of Science (Grant No.~20K14327). 
E.S. acknowledges financial supports from the Grant-in-Aid of 
Scientific Research (C) Japan Society for the Promotion of Science (Grant No.~19K03616) 
and Research Origin for Dressed Photon.

\appendix
\section{Meromorphic continuation of the resolvent}\label{A1}
We give another standard way to introduce resonances. 
We first compute an explicit expression of the resolvent for $\xi\in\C_+=\{z\in\C;\,\im z>0\}$, and continue it meromorphically to the lower half plane. Then resonances are the poles of this operator. 
In addition, the expression clearly shows that $\xi$ is a resonance if and only if $e^{-i\xi}$ is an eigenvalue of $\cK$. 
It means that this definition is equivalent to the one in Definition~\refup{Def:Res} (see Proposition~\refup{prop:EquivRes}).

\begin{proposition}
Assume \bu{(A1--2)}. For $f\in\cH$ and $\xi\in\C_+$, we define an operator $R(\xi)$ by
\begin{align*}
R(\xi)f(n):=
e^{-in\xi}P_0v(0)+\sum_{k=0}^{-n-1}e^{i(k+1)\xi}\bra{L}f(n+k)\ket{L}
+\sum_{k=0}^\infty e^{i(k+1)\xi}\bra{R}f(n-k)\ket{R},
\end{align*}
for $n<0$,
\begin{align*}
R(\xi)f(n):=v(n),
\end{align*}
for $n\in[n_0]$, and
\begin{align*}
R(\xi)f(n):=
\sum_{k=0}^\infty e^{i(k+1)\xi}\bra{L}f(n+k)\ket{L}
+e^{i(n-n_0)\xi}Q_{n_0}v(n_0)+\sum_{k=0}^{n-n_0-1}e^{i(k+1)\xi}\bra{R}f(n-k)\ket{R},
\end{align*}
for $n>n_0$, 
where we have set $v,\til{f}:[n_0]\to\C^2$ by 
\begin{align*}
&v=(e^{-i\xi}I-\cK)^{-1}\til{f},\\
&\til{f}(n)=f(n)+\dl_0(n)\sum_{k=0}^\infty e^{i(k+1)\xi}\bra{R}f(-k)\ket{R}
+\dl_{n_0}(n)\sum_{k=0}^\infty e^{i(k+1)\xi}\bra{L}f(n_0+k)\ket{L}.
\end{align*}
Then this is the $l^2$-resolvent, that is, we have $(e^{-i\xi}I-\cU)R(\xi)f=f$ and $R(\xi)f\in\cH$. 
Moreover, this resolvent operator $R(\xi)$ maps $\cH\cap\cH^\ope{out}_N$ to itself $(\forall N\ge0)$. 
\end{proposition}

Here, the $l^2$-boundedness of $R(\xi)$ is essentially shown by the facts that $\|g*h\|_{l^2}\le\|g\|_{l^2}\|h\|_{l^1}$ and $(e^{in\xi})_{n\ge0}\in l^1$.

\begin{corollary}
The operator $R(\xi)$ can be extended meromorphically for $\xi\in\C$ as an operator $\cH^\ope{out}\to\cH^\ope{out}$. 
Its poles coincide with the ones of $(e^{-i\xi} I-\cK)^{-1}$. In other words, $\xi$ is a pole of $R(\xi)$ if and only if $e^{-i\xi}$  is an eigenvalue of $\cK$. 
\end{corollary}

Kenta Higuchi, 
Department of Mathematical Sciences, 
Ritsumeikan University/ 
1-1-1 Noji-Higashi, Kusatsu, 
525-8577,  Japan

\textit{E-mail address}: ra0039vv@ed.ritsumei.ac.jp

Hisashi Morioka, 
Graduate School of Science and Engineering, 
Ehime University/ 
Bunkyo-cho 3, Matsuyama, Ehime, 
790-8577, Japan

\textit{E-mail address}: morioka@cs.ehime-u.ac.jp

Etsuo Segawa, 
Graduate School of Environment Information Sciences, 
Yokohama National University/ 
Hodogaya, Yokohama, 
240-8501, Japan

\textit{E-mail address}: segawa-etsuo-tb@ynu.ac.jp


\begin{thebibliography}{40}

%\bibitem[AgCo]
%{AgCo} 
%J.~Aguilar, J.M.~Combes\,:
%\newblock{\it A class of analytic perturbations for one-body Schr\"odinger Hamiltonians.}
%\newblock{Comm. Math. Phys.,} 22 (1971), no. 4, 269--279.
 
%\bibitem[Ash]
%{Ash} S.~Ashida\,: 
%\newblock{\it Molecular predissociation resonances below an energy level crossing. }
%\newblock{Asymptot. Anal.} 107 (2018), no. 3-4, 135--167.   

%\bibitem[As]
%{As} M.~Assal\,: 
%\newblock{\it Semiclassical resolvent estimates for Schr\"odinger operators with matrix-valued potentials and applications. }
%\newblock{work in progress.} 

%\bibitem[ADF]
%{ADF}
%M.~Assal, M.~Dimassi, S.~Fujii\'e\,:
%\newblock{\it Semiclassical trace formula and spectral shift function for systems via a stationary approach.}
%\newblock{Int. Math. Res. Notices,} (2019) 4, 1227--1264. 

%\bibitem[AsFu]
%{AsFu} M.~Assal, S.~Fujii\'e\,:
%\newblock{\it Eigenvalue splitting for a system of Schr\"odinger operators with an energy-level crossing. }
%\newblock{Preprint https://arxiv.org/pdf/1910.01195v2.pdf}
 
%\bibitem[Ba]
%{Ba} H.~Baklouti\,:
%\newblock{\it Asymptotique des largeurs de r\'esonances pour un mod\`ele d'effet tunnel microlocal.}
%\newblock{Ann. Inst. H. Poincare Phys. Theor.,} 68 (1998); no. 2, 179--228.

%\bibitem[BG]{BG}
%Buslaev, A.~Grigis\,:
%\newblock{\it  Imaginary part of Stark W ,}
%\newblock{ .}

%\bibitem[BY]
%{BY}
%M.S. Birman, D.R. Yafaev, 
%\newblock{\it  The spectral Shift function. The papers of M. G. Krein and their further development,}
%\newblock{ St. Petersbourg Math. J. 4 (1993) 833--870.}
 
%\bibitem[BFRZ1]
%{BFRZ1} J.-F.~Bony, S.~Fujii\'e, T.~Ramond, M.~Zerzeri\,:
%\newblock{\it Microlocal kernel of pseudodifferential operators at a hyperbolic fixed point.}
%\newblock{J. Funct. Anal.} 252 (2007) no. 1, 68--125.

%\bibitem[BFRZ2]
%{BFRZ2} J.-F.~Bony, S.~Fujii\'e, T.~Ramond, M.~Zerzeri\,:
%\newblock{\it Barrier-top resonances for non globally analytic potentials.}
%\newblock{J. Spectr. Theory,} 9 (2019), no. 1, 315--348.


%\bibitem[BP]
%{BP}
%V. Bruneau, V. Petkov, 
%\newblock{\it  Meromorphic continuation of the spectral shift function, }
%\newblock{ Duke Math. J. 116, (2003) 389--430.}

%\bibitem[Cdv]
%{Cdv} Y.~Colin de Verdi\'ere\,:
%\newblock{\it Bohr-Sommerfeld phases for avoided crossings.} 
%\newblock{Preprint} 2006, hal-00574571,  https://hal.archives-ouvertes.fr/hal-00574571.

%\bibitem[CdvPa]
%{CdvPa} Y.~Colin de Verdi\'ere, B.~Parisse\,:
%\newblock{\it \'Equilbre instable en r\'egime semi-classique:} I-Concentration microlocale.
%\newblock{Commun. In PDE, } 19 (9--10), 1535--1563 (1994).

%\bibitem[Ch]{Ch} T.~Christiansen\,:
%\newblock{\it Resonoance for steplike potentials: forward and inverse results.} 
%\newblock{Trans. Amer. Math. Soc.,} 358(5): 2071--2089 (2006).

%\bibitem[CK]{CK} A. Cohen, T. Kappeler, 
%\newblock{\it  Scattering and inverse scattering for steplike potentials in the Schr\"odinger equation, }
%\newblock{  (1983) .}

%\bibitem[CPS]{CPS}R.~D.~Costin, H.~Park, W.~Schlag\,:
%\newblock{\it The Weber equation as a normal form with applications to top of the barrier scattering}
%\newblock{J. Spectr. Theory,} 8, no. 2, 347--412 (2018).

%\bibitem[Di]
%{Di} M.~Dimassi\, :
%\newblock{\it  Spectral shift function and resonances for slowly varying perturbations of periodic Schr\"odinger operators.}
%\newblock{ J. Funct. Anal.}  225 (2005), p. 1--36.

%\bibitem[DiSj]{DiSj} M.~Dimassi, S.~Sj\"ostrand\,:
%\newblock{\it Spectral Asymptotics in the Semi-Classical Limit.}
%\newblock{Cambridge University Press,} 1999.

%\bibitem[DT]{DT} P.~Deift, E.~Trubowitz\,:
%\newblock{\it Inverse scattering on the line, }
%\newblock{Comm. Pure Appl. Math. }32 121--252.

%\bibitem[DyG]
%{DyG} 
%S.~Dyatlov, C.~Guillarmou\,: 
%\newblock{\it Scattering Phase Asymptotics with Fractal Remainders.}
%\newblock{ Communications in Mathematical Physics volume 324, (2013),  pages 425--444.}
 
\bibitem[DyZw]
{DyZw} S.~Dyatlov, M.~Zworski\,:
\newblock{\it Mathematical Theory of Scattering Resonances. Graduate Studies in Mathematics, 200.}
\newblock{American Mathematical Soc.,} 2019.
     
%\bibitem[Fe]
%{Fe} M.V.~Fedryuk\,:
%\newblock{\it Asymptotic Analysis,}
%\newblock{Springer-Verlag, Moscow, 1983.}

\bibitem[FeHi]
{FeHi} E.~Feldman, M.~Hillery\,:
\newblock{Quantum walks on graphs and quantum scattering theory,}
\newblock{\it Coding Theory and Quantum Computing,} 
edited by D.~Evans, J.~Holt, C.~Jones, K.~Klintworth, B.~Parshall, O.~Pfister, and H.~Ward, Contemporary Mathematics 381 (2005), pp71--96.

%\bibitem[FMW1]
%{FMW1} S.~Fujii\'e, A.~Martinez, T.~Watanabe\,:
%\newblock{\it Widths of resonances at an energy-level crossing I: Elliptic interaction.} 
%\newblock{J. Diff. Eq.}  260 (2016) 4051-4085.
     
%\bibitem[FMW2]
%{FMW2} S.~Fujii\'e, A.~Martinez, T.~Watanabe\,:
%\newblock{\it Widths of resonances at an energy-level crossing II: Vector field interaction.}
%\newblock{J. Diff. Eq.} 262 (2017) 5880-5895.

%\bibitem[FMW3]
%{FMW3} S.~Fujii\'e, A.~Martinez, T.~Watanabe\,:
%\newblock{\it Widths of resonances above an energy-level crossing.}
%\newblock{Preprint} https://arxiv.org/pdf/1904.12511.pdf

%\bibitem[FR]{FR} S.~Fujii\'e, T.~Ramond\,:
%\newblock{\it Matrice de scattering et r\'esonances associ\'ees \`a une orbite h\'et\'erocline.}
%\newblock{Annales de l'I.H.P. Physique th\'eorique,} Volume 69 (1998) no. 1, pp. 31-82.  

%\bibitem[Fu]{Fu} S.~Fujii\'e\,:
%\newblock{\it Semiclassical Representation of the Scattering Matrix by a Feynman Integral,}
%\newblock{Comm. in Math. Phys.,} 198 (1998), 407--425.

%\bibitem[GeMa]
%{GeMa} C.~G\'erard, A.~Martinez\,:
%\newblock{\it Principe d'absorption limite pour les op\'erateurs de Schr\"odinger a longue port\'ee.}
%\newblock{C. R. Acad. sci. Paris,} 306 (1988), 121--123.
 
%\bibitem[HeMa]
%{HeMa} B.~Helffer, A.~Martinez\,: 
%\newblock{\it Comparaison entre les diverses notions de r\'esonances.}
%\newblock{Helv. Phys. Acta,} 60 (1987), no. 8, 992-1003.
 
%\bibitem[HeSj]{HeSj} 
%B.~Helffer, J.~Sj\"ostrand\,:
%\newblock{\it R\'esonances en limite semiclassique.}
%\newblock{Bull. Soc. Math. France, M\'emoire No. 24-25,} 1986.
 
%\bibitem[HeSj2]
%{HeSj2} 
%B.~Helffer, J.~Sj\"ostrand\,:  
%\newblock{\it Semiclassical analysis for Harper's equation. III. Cantor structure of the spectrum. }
%\newblock{M\'em. Soc. Math. France}, no. 39 (1989), 1--124.

%\bibitem[Hi]
%{Hi}
%K.~Higuchi\,:
%\newblock{\it Resonance free domain for a system of Schr\"odinger operators with energy-level crossings. }
%\newblock{Preprint} https://arxiv.org/pdf/1912.10180.pdf

%\bibitem[Hi]{Hi}K.~Higuchi\,:
%\newblock{\it Resonant tunneling effect for a non-critical energy.}
%\newblock{work in progress}

\bibitem[Hi]{Hi}
K.~Higuchi\,:
\newblock{Feynman-type representation of the scattering matrix on the line via a discrete-time quantum walk,}
\newblock{\it J. Phys. A-Math. Theor.,} (2021), 33, 23.

\bibitem[HKKMS]{HKKMS}
K.~Higuchi, T.~Komatsu, N.~Konno, H.~Morioka, E.~Segawa\,:
\newblock{A discontinuity of the energy of quantum walk in impuritie,}
\newblock{\it Symmetry,} (2021), 13, 1134.


\bibitem[HKSS]{HKSS}
Y.~Higuchi, N.~Konno, I.~Sato, E.~Segawa\,:
\newblock{Quantum graph walks I: Mapping to quantum walks. }
\newblock{\it Yokohama Mathematical Journal,} no. 59 (2013), 33--55.


\bibitem[HiSe]{HiSe}
Y.~Higuchi, E.~Segawa\,:
\newblock{Dynamical system induced by quantum walks,}
\newblock{\it J. Phys. A-Math. Theor.,} 52 (2009), 395--202. 

%\bibitem[Iv]{Iv}V.~Ivrii\,:
%\newblock{\it Microlocal Analysis and Precise Spectral Asymptotics.}
%\newblock{Berlin: Springer}, (1998).

\bibitem[KKMS]{KKMS}
T.~Komatsu, N.~Konno, H.~Morioka, E.~Segawa\,:
\newblock{Generalized eigenfunctions for quantum walks via path counting approach,}
\newblock{\it Rev. Math. Phys.,} (2021), 33, 2150019.

%\bibitem[KMSW]
%{KMSW} 
%M.~Klein, A.~Martinez, R.~Seiler, X.W.~Wang\,:
%\newblock{\it  On the Born-Oppenheimer expansion for polyatomic molecules.}
%\newblock{Comm. Math. Physics,} 143 (1992), no. 3, 607-639.

%\bibitem[LTM]{LTM}S.~Y.~Lee, N.~Takigawa, C.~Marty\,:
%\newblock{\it A semiclassical study of optical potentials: Potential resonances.}
%\newblock{Nucl. Phys.} A308 (1978), 161--188.

%\bibitem[Ma1]
%{Ma1}
%A.~Martinez\,:
%\newblock{\it Estimates on complex interactions in phase space.}
%\newblock{Math. Nachr.} 167 (1994) 203--254. 

%\bibitem[Ma1]
%{Ma2}
%A.~Martinez\,:
%\newblock{\it Resonance free domains for non globally analytic potentials.}
%\newblock{Annales Henri Poincar\'e,} (2002), vol. 3, 739--756.

%\bibitem[Ma]{Ma} A.~Martinez\,:
%\newblock{\it An Introduction to Semiclassical and Microlocal Analysis.}
%\newblock{Springer-Verlag New-York, UTX Series,} 2002.

%\bibitem[Me]
%{Me} R. ~Melrose\,: 
%\newblock{\it  Weyl asymptotics for the phase in obstacle scattering,}
%\newblock{  Comm. Partial Differential Equations 13 (1988) 1431--1439.}

\bibitem[MMOS]{MMOS}
K.~Matsue, L.~Matsuoka, O.~Ogurisu, E.~Segawa\,:
\newblock{Resonant-tunneling in discrete-time quantum walk.}
\newblock{\it Quantum Stud.: Math. Found.} 6, 35 44 (2019). 

\bibitem[Mo]{Mo}
H.~Morioka\,:
\newblock{Generalized eigenfunctions and scattering matrices for position-dependent quantum walks,}
\newblock{\it Rev. Math. Phys.,} 31 (2019), 1--37.

%\bibitem[Na]
%{Na} S.~Nakamura\,:  
%\newblock{\it On an example of phase-space tunneling.} 
%\newblock{Ann. Inst. H. Poincare Phys. Theor.,}  63 (1995), no. 2, 211-229.

%\bibitem[Ne]
%{Ne} L.~N\'ed\'elec\,: 
%\newblock{\it  Localization of resonances for matrix Schr\"odinger operators, }
%\newblock{ Duke Math. J. 106 (2) (2001) 209-236.}
 
%\bibitem[Pe]
%{Pe} P.~Pettersson\,:
%\newblock{\it WKB expansions for systems of Schr\"odinger operators with crossing eigenvalues.}
%\newblock{Asymptotic Analysis,} 14 (1997) 1--48.

%\bibitem[Ra]{Ra} T.~Ramond\,:
%\newblock{\it Semiclassical study of quantum scattering on the line.}
%\newblock{Commun. Math. Phys.} 177, 221--254 (1996).
   
%\bibitem[ReSi]{ReSi} M.~Reed, B.~Simon\,:
%\newblock{\it Methods of Modern Mathematical Physics, Vol. I-IV.}
%\newblock{Academic Press New York}, 1972.

%\bibitem[Se]{Se} E.~Servat\,:
%\newblock{\it  R\'esonances en dimension un pour l'op\'erateur de Schr\"odinger,}
%\newblock{Asymptotic Anal. 39 (2004), 187--224.}


%\bibitem[Sj]
%{Sj} J.~Sj\"ostrand, 
%\newblock{\it  A trace formula and review of some estimates for resonances, Microlocal Analysis and Spectral Theory,}
%\newblock{  Kluwer, NATO ASI Series C, vol. 490, 1997, pp. 377--437.}

%\bibitem[Sj1]
%{Sj1} 
%J.~Sj\"ostrand\,: 
%\newblock{\it Density of states oscillations for magnetic Schr\"odinger operators. } 
%\newblock{Mathematics in Science and Engineering,} Volume 186, 1992, Pages 295--345.

%\bibitem[Sj2]
%{Sj2} 
%J.~Sj\"ostrand\,: 
%\newblock{\it Projecteurs adiabatiques du point de vue pseudodiff\'erentiel. } 
%\newblock{C. R. Acad. Sci. Paris S\'er. I Math.} 317 (1993), no. 2, 217--220.

\bibitem[RST]{RST}
S.~Richard, A.~Suzuki, and R.~Tiedra de Aldecoa\,:
\newblock{``Quantum walks with an anisotropic coin II: scattering theory",}
\newblock{\it Lett. Math. Phys.,} \textbf{109} 1 (2019), 61--88. 

%\bibitem[SjZw]
%{SjZw} J.~Sj\"ostrand, M.~Zworski\,:
%\newblock{\it Fractal upper bounds on the density of semiclassical resonances.}
%\newblock{Duke Math. J.} 137, no. 3 (2007), 381--459.

\bibitem[Su]{Su}
A.~Suzuki\,:
\newblock{``Asymptotic velocity of a position-dependent quantum walk",}
\newblock{\it Quantum Information Processing.} Jan2016, Vol. 15 Issue 1, p103-119. 

\bibitem[Ta]{Ta}
G.K.~Tanner\,:
\newblock{``From quantum graphs to quantum random walks",}
\newblock{\it Nato. Sci. Ser. II Math.,} 2006, p69-87. 


\bibitem[Ti]{Ti}
R.~Tiedra de Aldecoa\,:
\newblock{``Stationary scattering theory for unitary operators with an application to quantum walks",}
\newblock{\it J. Funct. Anal.,} (2020), 279, 108704. 

\bibitem[Ya]
{Ya} D.~Yafaev\,: 
\newblock{\it Mathematical Scattering Theory: General Theory,} \newblock{Translations of Mathematical Monographs,} Vol. 105 American Mathematical Soc., Providence, RI, (2009).

%\bibitem[Ya2]
%{Ya2} D.R.~Yafaev\,: 


%\bibitem[Zw]
%{Zw} M.~Zworski\,:
%\newblock{\it Semiclassical Analysis. Graduate Studies in Mathematics, 138.}
%\newblock{American Mathematical Soc.,} 2012.

\end{thebibliography}
\end{document}